\documentclass[12pt,A4]{article}
\usepackage[english]{babel}
\usepackage{amsmath,amsfonts,amssymb,amsthm,mathrsfs,bbm}
\usepackage[font=sf, labelfont={sf,bf}, margin=1cm]{caption}
\usepackage{graphicx,graphics}
\usepackage{epsfig}
\usepackage{latexsym}
\usepackage[applemac]{inputenc}
\usepackage{ae,aecompl}
\usepackage{pstricks}
\usepackage{enumerate}
\usepackage{xcolor}

\usepackage[pdfpagemode=UseNone,bookmarksopen=false,colorlinks=true,urlcolor=blue,citecolor=blue,citebordercolor=blue,linkcolor=blue]{hyperref}

\usepackage[normalem]{ulem}
\usepackage[top=2.5cm,left=1.5cm,right=1.5cm,bottom=2.5cm]{geometry}
\usepackage{comment}
\usepackage{eulervm,palatino,stmaryrd,comment}

\newcommand\Es[1]{\mathbb{E}\left[#1\right]}
\renewcommand\Pr[1]{\mathbb{P}\left(#1\right)}
\newcommand\Esx[1]{\mathbb{E}_{x}\left[#1\right]}
\newcommand\Prx[1]{\mathbb{P}_{x}\left(#1\right)}
\newcommand\EsQ[1]{\mathbb{Q}\left[#1\right]}

\newcommand\EsQx[1]{\mathbb{Q}_{x}\left[#1\right]}

\def \N {\mathbb N}
\def \R {\mathbb R}

\def \E {\mathbb E}
\def \la {\lambda}
\def \P {\mathbb{P}}
\def \Q {\mathbb{Q}}
\def \Exp {\textnormal{\textsf{Exp}}}
\newtheorem{thm}{Theorem}

\newtheorem{prop}[thm]{Proposition}
\newtheorem{cor}[thm]{Corollary}

\theoremstyle{definition}

\title{  \vspace {-2cm}\textbf{Predator-prey dynamics on infinite trees: a branching random walk approach}}
\date{}
\author{Igor Kortchemski\thanks{DMA, École Normale Supérieure, E-mail: igor.kortchemski@normalesup.org}}

\DeclareSymbolFont{extraup}{U}{zavm}{m}{n}
\DeclareMathSymbol{\varheart}{\mathalpha}{extraup}{86}
\DeclareMathSymbol{\vardiamond}{\mathalpha}{extraup}{87}

\makeatletter
\renewcommand*{\@fnsymbol}[1]{\ensuremath{\ifcase#1\or  \vardiamond \or \clubsuit\or \spadesuit\or
   \mathsection\or \mathparagraph\or \|\or **\or \dagger\dagger
   \or \ddagger\ddagger \else\@ctrerr\fi}}
\makeatother

\begin{document}

\maketitle

\let\thefootnote\relax\footnotetext{ \\\emph{MSC2010 subject classifications}. 60J80,60J10,60F05. }

\begin{abstract} 
We are interested in predator-prey dynamics on infinite trees, which can informally be seen as particular two-type branching processes where individuals may die (or be infected) only after their parent dies (or is infected). We study two types of such dynamics: the chase-escape process, introduced by Kordzakhia with a variant by Bordenave, and the birth-and-assassination process, introduced by Aldous \& Krebs. We exhibit a coupling between these processes and branching random walks killed at the origin. This sheds new light on the chase-escape  and  birth-and-assassination processes, which allows us to recover by probabilistic means previously known results and also to obtain new results. For instance, we find the asymptotic behavior of tail of the number of infected individuals in both the subcritical and critical regimes for the chase-escape process, and show that the birth-and-assassination process ends almost surely at criticality. 
\end{abstract}

\section{Introduction}

We study the asymptotic behavior of two predator-prey dynamics on infinite trees.  Let us first give an informal description of the chase-escape process on regular trees, which has been introduced by Kordzakhia \cite {Kord05}. Vertices can be of three types: predators, preys or vacant. At fixed rate $ \lambda>0$, preys may only spread to vacant nearest neighbors (when a prey spreads to a vacant neighbor, both vertices are then preys), while at fixed rate $1$, predators may only spread to either vacant or prey nearest neighbors. The evolution starts with one predator at the root and one neighboring prey. Kordzakhia \cite {Kord05} identified the critical value of $ \lambda= \lambda_{c}$ that allows the preys to survive with positive probability. Later, a variant has been considered by Bordenave \cite {Bor08}, where this time predators may only spread to nearest neighbors occupied by preys. In this context, for a certain class of trees, Bordenave \cite{Bor14} extended Kordzakhia's result, estimated the probability of extinction of the preys for $ \lambda> \lambda_{c}$, and studied the moments of the final total number of predators. 

Let us now informally describe the birth-and-assassination process, which was introduced by Aldous \& Krebs \cite {AK90}, and is a system of evolving individuals. During its lifetime, each individual produces offspring at fixed rate $ \mu>0$. In addition, each individual $u$ is equipped with a random timer of $K_{u}$ units of times, where the positive random variables $( \mathcal{K} _{u})$ are all independent identically distributed. The individual's timer is triggered as soon as its parent dies, and an individual dies when its timer reaches $0$. The evolution starts with one individual with its timer triggered. Aldous \& Krebs give the image of a finite collection of clans (which are the connected components of the living individuals in the genealogical trees) where only the heads of clans can be killed. Under certain conditions on $ \mathcal{K} _{u}$, Aldous \& Krebs identified the critical value of $ \mu= \mu_{c}$ that allows the evolution to survive indefinitely with positive probability (the motivation of Aldous \& Krebs was to analyze a scaling limit of a queueing process with blocking, arising in database processing \cite {TPH86}). Later, Bordenave \cite {Bor08,Bor14} studied the birth-and-assassination process when $K_{u}$ is an exponential random variable, and Bordenave proved in particular that the evolution stops at $ \mu= \mu_{c}$ with probability $1$.

In this work, we exhibit a coupling of these two models with  branching random walks killed at the barrier $0$, thus shedding new light on these models and giving an interesting application of the theory of killed branching random walks which do not have independent displacements. In the case of the chase-escape model on supercritical Galton--Watson trees, this enables us to make use of recent powerful results concerning killed branching random walks \cite{AHZ13+} and extend (under slightly more restrictive conditions on the structure of the tree) several results of Bordenave obtained in an analytic way and also to establish new results by probabilistic arguments. In particular, for $ \lambda \leq  \lambda_{c}$, we find the asymptotic behavior of the tail of the final number of predators. Bordenave \cite{Bor14} observed that the chase-escape process exhibits heavy-tail phenomena similar to those appearing in the Brunet-Derrida model of branching random walk killed below a linear barrier. Our coupling shows that this is not a coincidence.  In the case of the birth-and-assassination process, we also show that in general case the evolution stops at $ \mu= \mu_{c}$ with probability $1$, without assuming that $K_{u}$ is an exponential random variable.

\paragraph{The chase-escape process.} Following  Bordenave \cite{Bor14}, we now give a formal definition of the chase-escape process. For convenience of notation, we shall define it as a SIR (Susceptible, Infected, Recovered) dynamics, where preys (resp.~predators and vacant sites) are Infected (resp.~Recovered and Susceptible) sites. As in  \cite{Bor08}, one may also think of Susceptible/Infected/Recovered individuals as normal individuals/individuals propagating a rumor/individuals trying to scotch it.

Let $G=(V,E)$ be a locally finite connected graph.  Set $ \mathfrak {X}= \{ S,I,R\}^{V}$ and for every $v \in V$, let $I_{v}, R_{v}:  \mathfrak {X}\rightarrow  \mathfrak {X}$ be the maps defined by $(I_{v}(x))_{u}=(R_{v}(x))_{u}=x_{u}$ if $u \neq v$ and $(I_{v}(x))_{v}=I$ and $(R_{v}(x))_{v}=R$ with $x=(x_{u})_{u \in V}$. The chase-escape process of infection intensity $ \lambda>0$ is a Markov process taking values in $ \mathfrak {X}$ with transition rates
$$Q(x,I_{v}(x))= \lambda  \cdot \mathbbm {1}_{ \{ x_{v}=S\} }  \cdot  \sum_{ \{ u,v\} \in E } \mathbbm {1}_{ \{ x_{u}=I\} }, \quad Q(x,R_{v}(x))= \mathbbm {1}_{ \{ x_{v}=I\} }  \cdot  \sum_{\{ u,v\} \in E} \mathbbm {1}_{ \{ x_{u}=R\} }  \qquad (v \in V, x \in \mathfrak {X}).$$  
This means that infected vertices spread with rate $ \lambda$ to neighboring susceptible vertices and that recovered vertices spread with rate $1$ to neighboring infected vertices. Note that up to a time change there is no loss in generality in supposing that recovered vertices spread at rate $1$. This dynamics differs from the classical SIR epidemics model, where infected vertices recover at a fixed rate (not depending on their neighborhood). Earlier, Kordzakhia \cite {Kord05} has considered a similar model where recovered vertices may spread with rate $1$ to neighbors who are are either infected or recovered. We say that the chase-escape process gets extinct if there exist a (random) time at which there are no more infected vertices.

If there are no recovered vertices present in the beginning, this dynamics is the so-called Richardson's model \cite {Ric73}. Let us also mention that H\"aggstr\"om \& Pemantle  \cite {HP98} and Kordzakhia  \& Lalley \cite {KL05} have studied an  extension of Richardson's model with two species, in which infected and recovered vertices may only spread to susceptible vertices (in particular, infected and recovered vertices never change states). See also \cite{Bor08} for a directed version of the chase-escape process called the rumor-scotching process, and \cite{Kor13b} for a study of the chase-escape process on large complete graphs.

In this work, we will be interested in the chase-escape process on (rooted) trees. To describe the initial condition which we will use, we introduce the following notation. First, let $ \mathcal{T}$ be a tree. If $A \subset \mathcal{T}$ is a subset of vertices and $x \geq 0$, we denote by $ \textsf{C}( \mathcal{T}, \mathcal{A} ,x)$ the chase-escape process on $ \mathcal{T}$ starting with the vertices of $ \mathcal{A}$ being Infected and all the other vertices being Susceptible, and with a Recovered vertex attached to the root of $ \mathcal{T}$ (denoted by $ \emptyset$) at time $x$. To simplify, we shall say that  $ \textsf{C}( \mathcal{T},A,x)$ is the chase-escape process with initial condition $ \mathcal{A}$ and delay $x$.  In particular, if $ \overline{ \mathcal{T} }$ is the tree obtained from $ \mathcal{T}$ by adding a new parent to the root of $ \mathcal{T}$ and by rooting $\overline{ \mathcal{T} }$ at this new vertex, notice that the chase-escape process $ \textsf{C}( \mathcal{T},  \{ \emptyset\},0)$ may be seen as the chase-escape starting with the root of $ \overline{  \mathcal{T} }$ being Recovered, its child being Infected and all the other vertices being Susceptible. This is the initial condition considered in  \cite{Bor14}.

\paragraph {Coupling the chase-escape process with branching random walks.} We now introduce some background concerning one-dimensional discrete branching random  walks on the real line $ \R$. Let $ \mathcal{T}$ be a tree rooted at $ \emptyset$ and $ x \in \R$. If $u \in \mathcal{T}$, let $ \llbracket \emptyset,u \rrbracket$ denote the vertices belonging to the shortest path connecting $ \emptyset$ to $u$, set $ \rrbracket \emptyset, u \rrbracket = \llbracket \emptyset,u \rrbracket \backslash \{  \emptyset\}$ and let $|u|= \#  \rrbracket \emptyset, u \rrbracket$ be the generation of $u$. Let $ ( \xi_{u}; u \in \mathcal{T})$ be a collection of independent and identically distributed (i.i.d.) random variables. Then set
\begin{equation}
\label{eq:defBRW} V(u) := x+ \sum_{ v \in \rrbracket \emptyset, u \rrbracket } \xi_{v}, \qquad u \in \mathcal{T}.
\end{equation} 
The collection $(V(u); u \in \mathcal{T})$ will be called a branching random walk on $ \mathcal{T}$ with displacement distribution $ \xi_{1}$, starting from $x$. If we view the tree $ \mathcal{T}$ as a genealogical tree and give each individual a displacement by saying that when a vertex $u$ is born, its displacement is obtained by adding $ \xi_{u}$ to the displacement of its parent (with the displacement of the root being $x$), then the displacement of $u$ is exactly $V(u)$.

We are now ready to present the link between the chase-escape process on trees and branching random walks. Here and later, $ \Exp( \lambda)$ denotes an exponential random variable of parameter $ \lambda>0$, independent of all the other mentioned random variables (in particular, different occurrences of $\Exp(1)$ denote different independent random variables).

Fix a tree $ \mathcal{T}$. Let $(R(u); {u \in \mathcal{T} })$ be a branching random walk on $ \mathcal{T}$ with displacement distribution $\Exp(1)$, starting from a random  point distributed according to an (independent) $\Exp(1)$ random variable. Let also $(I(u); {u \in \mathcal{T} })$ be a branching random walk on $ \mathcal{T}$ with displacement distribution $ \Exp (\lambda)$, starting from $0$. Finally, for every $u \in \mathcal{T}$, let $ \overleftarrow{u}$ be the parent of $u$ and set $W(u)= R(\overleftarrow{u})-I(u)$, with the convention $R(\overleftarrow{ \emptyset})=0$.
 
\begin{thm}[Coupling for the chase-escape process] \label {thm:coupl1}  For every $t \geq 0$ and $ u \in \mathcal{T}$,  set:
$$X_{t}(u)= \begin {cases} I & \textrm { if } I(u) \leq  t < R({u}) \textrm { and } W(v) \geq 0 \textrm { for every } v \in \llbracket \emptyset,u\rrbracket, \\
 R & \textrm { if } R({u}) \leq   t \textrm { and } W(v) \geq 0 \textrm { for every } v \in \llbracket \emptyset,u\rrbracket,  \\
S & \textrm {otherwise}.
\end {cases}$$
Then $(X_{t})_{t \geq 0}$ has the same distribution as the  chase-escape process $ \textsf{C}( \mathcal{T},  \{ \emptyset\},0)$.
\end {thm}

Intuitively, if we let $ \emptyset =u_{0},u_{1}, \ldots, u_{|u|}=u$ be the vertices belonging to the shortest path connecting $ \emptyset$ to $u \in \mathcal{T} $, then $u$ becomes infected if and only if, for every $1 \leq i \leq |u|$,  $u_{i}$ has been infected before its parent has recovered, meaning that $ W(v) \geq 0 \textrm { for every } v \in \llbracket \emptyset,u\rrbracket$. In this case, $I(u)$ is the first time when $u$ becomes infected, while $R({u})$ is the first time $u$ recovers. Theorem \ref {thm:coupl1} is then a simple consequence of the memoryless property of the exponential distribution. We omit a formal proof, since it would not be enlightening. This coupling, based on two independent branching random walks, is reminiscent of the Athreya--Karlin coupling \cite {AK68} used to study urn schemes by introducing independent branching processes.

Without further notice, we will assume that the chase-escape process is the process $(X_{t})_{t \geq 0}$ constructed in Theorem \ref {thm:coupl1}.

 \paragraph {Coupling the final state of the chase-escape process with killed branching random walks.} By taking the limit $t \rightarrow \infty$ in Theorem \ref{thm:coupl1}, we get that the total number of infected individuals in the chase-escape process is equal to the number of elements of  $ \{  u \in \mathcal{T};  W(v) \geq 0 \textrm { for every } v \in \llbracket \emptyset,u\rrbracket\}$. Unfortunately $W$ is \emph{not} a branching random walk in the sense defined by \eqref{eq:defBRW}.  However, we still manage to represent the number of infected individuals as the total progeny of a killed branching random walk by allowing branching random walks with non independent increments.
  
 More precisely, consider a single individual located at $x \in \R$. Its children are positioned on $ \R$ according to a certain point process $ \mathcal{L}$ on $ \R$ and represent the first generation. Each one of the individuals of the first generation independently give birth to new individuals positioned (with respect to their birth places) according to an independent point process having the same law as $ \mathcal{L}$, which represent the second generation. This process continues similarly for the next generation.  The genealogy of the individuals form a Galton--Watson tree denoted by $ \mathcal{T}$ and for every $u \in \mathcal{T}$ we let $V(u)$ be the position (or displacement) of $u$. We say that the collection $(V(u); u \in \mathcal{T})$ is a branching random walk with displacement distribution given by the point process $  \mathcal{L}$, starting from $x$.  We also introduce a killing barrier at the origin: we imagine that any individual entering $(- \infty,0)$ is immediately killed and that its descendance is removed. Therefore, at every generation $n \geq 0$, only the individuals with a displacement that has always remained nonnegative up to generation $n$ survive. In particular, $\{  u \in \mathcal{T} ; V(v) \geq 0, \forall v \in \llbracket \emptyset, u \rrbracket \}$ is the set of all the individuals that survive.

\begin{cor}\label {cor:coupl1} Let $V$ be the branching random walk  produced with the point process $$    \sum_{i=1} ^ U \delta_{ \{  \mathcal{E}-\Exp_{i}( \lambda)\} }, $$
starting from $x$, where $U$ is a non-negative integer valued random variable, where $ \mathcal{E}$ is an independent exponential random variable of parameter $1$ and $ (\Exp_{i}( \lambda))_{i \geq 1}$ is an independent sequence of i.i.d.~exponential random variables of parameter $ \lambda$. Denote by $ \mathcal{T}$ the genealogical tree of this branching walk,  set $Z_{n}= \# \{  u \in \mathcal{T}; \ |u|=n \textrm { and }  V(v) \geq 0 \textrm { for every } v \in \llbracket \emptyset,u\rrbracket\}$ and  $Z= \sum_{n \geq 0} Z_{n}$. Then $Z_{n}$ has the same distribution as the number of individuals at generation $n$ in $ \mathcal{T}$ that have been infected until the chase-escape process  $ \textsf{C}( \mathcal{T},  \{ \emptyset\},x)$ has reached its absorbing state, and $Z$ has the same distribution as the total number of individuals  that have been infected in $ \textsf{C}( \mathcal{T},  \{ \emptyset\},x)$.
\end{cor}

This easily follows from Theorem \ref {thm:coupl1} by observing that the displacement $V(u)$ is the difference between the time when the parent of $u$ recovers and the time when $u$ gets infected (without the presence of the recovered vertices). In future work, we hope to use similar couplings to study the chase-escape process on different kind of graphs such as $ \mathbb{Z}^2$.

\paragraph {Chase-escape processes on Galton--Watson trees.} Our results concern the asymptotic behavior of the chase-escape process on supercritical Galton--Watson trees. Let $ \nu$ be a probability measure on $ \mathbb{Z}_{+}= \{ 0,1,2, \ldots\}$. Recall that a Galton--Watson tree with offspring distribution $\nu$ is a random rooted tree starting from one individual by supposing that each individual has an i.i.d.~number of offspring distributed according to $ \nu$ (see e.g.~\cite[Section 1]{LG05} for a formal definition). If $d \geq 2$ is an integer such that $ \nu(d)=1$,  $ \mathcal{T}$ is the infinite deterministic $d$-ary tree.

 We henceforth assume that $ \mathcal{T}$ is a Galton--Watson tree with offspring distribution $ \nu$, and are interested in the properties of the chase-escape process on $ \mathcal{T}$. Let $d= \sum_{i \geq 0}i \nu(i)$ be the expected number of offspring of an individual and suppose that $d  \in (1, \infty)$, which implies that $ \# \mathcal{T}= \infty$ with positive probability. 

If $ d \geq 2$ is an integer, Kordzakhia \cite {Kord05} identified $$\lambda_{c} := 2d-1- 2 \sqrt {d(d-1)}$$ as the critical value of $ \lambda$ by showing that for $ \lambda> \lambda_{c}$, the probability that the chase-escape process $ \textsf{C}( \mathcal{T},  \{ \emptyset\},0)$ gets extinct on the infinite $d$-ary tree is less than one, and that for $ \lambda \in (0, \lambda_{c})$ the latter probability is equal to one. Bordenave \cite[Theorem 1.1]{Bor14} extended this result to more general trees, such as supercritical Galton--Watson trees. By using heavy analytic tools, Bordenave treats the case $ \lambda= \lambda_{c}$ for Galton--Watson trees:  

\begin{prop}[Bordenave, Corollary 1.5 in \cite{Bor14}]\label {prop:crit} For $ \lambda = \lambda_{c}$, for almost every $ \mathcal{T}$,  almost surely, the chase-escape process $ \textsf{C}( \mathcal{T},  \{ \emptyset\},0)$ gets extinct.
\end{prop}

We give a very short probabilistic proof of Proposition \ref {prop:crit} based on Corollary \ref {cor:coupl1} and on the almost sure convergence towards $0$ of the Biggins' martingale associated with our branching random walk (see Section \ref {sec:brw}).

Our first main result gives the asymptotic behavior of the tail of the number of infected individuals in the chase-escape process.

\begin {thm} \label {thm:tail} Denote by $Z$ the total number of vertices that have been infected until the chase-escape process  $ \textsf{C}( \mathcal{T},  \{ \emptyset\},0)$ has reached its absorbing state. Let $U$ denote a random variable distributed according to $ \nu$. 
\begin{enumerate}
\item[(i)] Suppose that $ \lambda= \lambda_{c}$ (critical case). Assume that there exists $ \alpha>2$ such that $ \Es {U^{ \alpha}}< \infty$. Then
$$ \P( Z >n) \quad  \mathop {\sim}_{n \rightarrow \infty} \quad  \left( 1+ \sqrt{\frac{ d}{d-1}} \right)  \cdot \frac{1}{n ( \ln(n))^{2}}.$$
\item[(ii)] Suppose that  $ \lambda \in (0, \lambda_{c})$ (subcritical case). Assume that there exists $ \alpha> \frac{(1- \lambda+ \sqrt {\lambda^{2}-2 \lambda (2d-1)+1})^{2}}{2(d-1) \lambda}$ such that $ \Es {U^{ \alpha}}< \infty$. There exists a constant $C_{1}>0$, depending only on $ \lambda$ and $d$, such that 
$$\P( Z >n) \quad  \mathop {\sim}_{n \rightarrow \infty} \quad C_{1} \cdot n^{- \frac{(1- \lambda+ \sqrt {\lambda^{2}-2 \lambda (2d-1)+1})^{2}}{4(d-1) \lambda}}.$$
\end{enumerate}
\end {thm}

Theorem \ref {thm:tail} is a consequence of analogous results obtained by Aïdékon, Hu \& Zindy \cite[Theorem 1]{AHZ13+} for killed branching random walks. In particular, in the critical case, note that $ \Es {Z}< \infty$ while $ \Es {Z \ln(Z)} = \infty$.
Let us mention that Theorem \ref {thm:tail} (ii) is consistent with a result of Bordenave \cite[Theorem 1.3]{Bor14}, which states that, for every supercritical offspring distribution $ \nu$,
$$ \sup \{ u \geq 1; \Es {Z^{u}}< \infty\}= \min \left( \frac{(1- \lambda+ \sqrt {\lambda^{2}-2 \lambda (2d-1)+1})^{2}}{4(d-1) \lambda},\sup \{ u \geq 1; \Es {U^{u}}< \infty\}  \right).$$

The fact that the tail of number of infected individuals in the chase-escape process $ \textsf{C}( \mathcal{T},  \{ \emptyset\},0)$ is regularly varying allows us to obtain analogous results for more general initial conditions, where we start the chase-escape process with any finite number of infected individuals. To state these results we denote by $ \partial A$ the vertices belonging to $ \mathcal{T} \backslash \mathcal{A}$ and at distance $1$ from $ \mathcal{A}$. Recall also that $|u|$ denotes the generation of a vertex $u \in \mathcal{A}$.

\begin {thm} \label {thm:tail2} Assume that $ \nu(d)=1$ where $d \geq 1$ is an integer, so that  $ \mathcal{T}$ is the infinite $d$-ary tree. Let $ \mathcal{A} \subset \mathcal{T}$ be a finite connected subset of $ \mathcal{T}$ containing $ \emptyset$. Denote by $Z_{ \mathcal{A} }$ the total number of vertices that have been infected until the chase-escape process  $ \textsf{C}( \mathcal{T},  \mathcal{A} ,0)$ has reached its absorbing state. 
\begin{enumerate}
\item[(i)] Suppose that $ \lambda= \lambda_{c}$. Then
$$ \P( Z_{ \mathcal{A} } >n) \quad  \mathop {\sim}_{n \rightarrow \infty} \quad  \left( 1+ \sqrt{\frac{ d}{d-1}} \right)  \cdot \left(   \sum_{u \in \partial A}  \frac{|u|}{d \cdot (d- \sqrt {d(d-1)})^{|u|-1}} \right)  \cdot \frac{1}{n ( \ln(n))^{2}}.$$
\item[(ii)] Suppose that  $ \lambda \in (0, \lambda_{c})$. Then
$$\P( Z_{ \mathcal{A} } >n) \quad  \mathop {\sim}_{n \rightarrow \infty} \quad     C_{1} \cdot  \frac{ \lambda   \sum_{u \in \partial A}  ( (\rho_{-}+ \lambda)^{-|u|}- (\rho_{+}+ \lambda)^{-|u|})  }{ \sqrt {\lambda^{2}-2 \lambda (2d-1)+1}}  \cdot  n^{- \frac{(1- \lambda+ \sqrt {\lambda^{2}-2 \lambda (2d-1)+1})^{2}}{4(d-1) \lambda}},$$
where $ \rho_{-}$ and $ \rho_{+}$ are defined in \eqref{eq:rho}.
\end{enumerate}
\end {thm}

Theorem \ref {thm:tail2} also  relies on results of \cite{AHZ13+}, but its proof uses estimates for killed branching random walks starting from a random point. In particular, the main difficulty to obtain Theorem \ref {thm:tail2} is to establish uniform estimates in the starting point (see Proposition \ref {prop:tails}, which may be of independent interest). Many constants are explicit in Theorem \ref {thm:tail2}: This stems from the remarkable fact that the renewal function of the random walk associated with the branching random walk $V$ may be calculated explictly.

Observe finally that if $0 \in \mathcal{A}$ and $ \mathcal{A}$ is not a connected subset of $ \mathcal{T}$, then $ \Pr {Z_{ \mathcal{A}}= \infty}>0$. Indeed, if $u \in \mathcal{A}$ is such that $ \overleftarrow{u} \not \in \mathcal{A}$, then  $\overleftarrow{u}$ has a positive probability of never becoming infected and thus of never recovering, so that all the descendants of $u$ will be infected with positive probability.

\paragraph{Birth-and-assassination process.} We now turn our attention to the birth-and-assassination process, which we first formally define following Aldous \& Krebs \cite {AK90}. Let $\N$ be the set of all the positive integers and  let $ \mathcal{U}$ be the set of all labels defined by $\mathcal{U} = \bigcup_{n=0}^{\infty}(\N)^n$, 
where by convention $(\N)^0 = \{\emptyset\}$. An element of $\mathcal{U}$ is a
sequence $u=u_1\cdots u_k$ of positive integers and we set
$|u| = k$, which represents the  generation  of $u$. If $u = u_1
\cdots u_i$ and $v = v_1 \cdots v_j$ belong to $\mathcal{U}$, we write $uv= u _1
\cdots u_i v_1 \cdots v_j$ for the concatenation of $u$ and $v$. In
particular, we have $u \emptyset = \emptyset u = u$. Let $(\mathcal{P}_{u} ; u \in \mathcal{U})$ be a family of i.i.d.~Poisson processes on $ \R_{+}$ with common arrival rate $ \lambda>0$, and let $ ( \mathcal{K}_{u} ; u \in \mathcal{U})$ be an independent collection of i.i.d.~strictly positive random variables.  

The process starts at time $0$ with one individual, with label $ \emptyset$, producing offspring at arrival times of $ \mathcal{P}_{ \emptyset}$, which enter the system with labels $1,2, \ldots$  according to their birth order. Each new individuel $u$ entering the system immediately begins to produce offspring at the arrival times of $ \mathcal{P}_{u}$, which enter the system with labels $n1,n2, \ldots$ according to their birth order. In addition, the ancestor $ \emptyset$ is \emph{at risk} at time $0$, and it only produces offspring until time $T_{ \emptyset} := \mathcal{K} _{ \emptyset}$, when it is removed from the system. In addition, if an individual $u$ is removed from the system at time $T_{u}$, then, for  every $k \geq 0$, such that the individual $uk$ has been born,  $uk$ becomes at risk, and continues to produce offspring until time $T_{uk}:=T_{u}+ \mathcal{K}_{uk} $ when it is removed from the system.

The birth-and-assassination process can be equivalently  seen as a Markov process on $ \{ S,I,R\}^{\mathcal{U}}$, where an individual at state R (resp.~I and S) is a removed (resp.~alive and not yet born) individual.

We will couple the birth-and-assassination process with a branching random walk where individuals have infinitely many offspring. Let $(D(u) ; {u \in \mathcal{U} })$ be the branching random walk on $ \mathcal{U}$ with displacement distribution given by the point process formed by a collection of i.i.d.~variables having the same distribution $K_{ \emptyset}$, starting from an independent  starting point distributed according to $K_{ \emptyset}$. Let also $(B(u); {u \in \mathcal{U} })$ be the branching random walk on $ \mathcal{U}$ with displacement distribution given by the point process $$   
 \delta_{ \{ \Exp_{1}( \lambda)\} }+    \delta_{ \{ \Exp_{1}( \lambda)+\Exp_{2}( \lambda)\} } +  \delta_{ \{ \Exp_{1}( \lambda)+\Exp_{2}( \lambda)+\Exp_{3}( \lambda)\}} +\cdots, $$
where $(\Exp_{i}( \lambda))_{i \geq 1}$ is an i.i.d.~collection of exponential random variables of parameter $ \lambda$, with starting point $0$. Finally set $V'(u)= D(\overleftarrow{u})-B(u)$ for $u \in \mathcal{T}$, with the convention $D(\overleftarrow{\emptyset})=0$.

\begin{thm}[Coupling for the birth-and-assassination process]\label {thm:coupl2}For every $t \geq 0$ and $ u \in \mathcal{U}$, we set:
$$Y_{t}(u)= \begin {cases} I & \textrm { if } B(u) \leq  t < D({u}) \textrm { and } V'(v) \geq 0 \textrm { for every } v \in \llbracket \emptyset,u\rrbracket \\
 R & \textrm { if } D({u}) \leq   t \textrm { and } V'(v) \geq 0 \textrm { for every } v \in \llbracket \emptyset,u\rrbracket \\
 S & \textrm { otherwise.}
 \end {cases}$$
Then $(Y_{t})_{t \geq 0}$ has the same distribution as the birth-and-assassination process.
\end {thm}

This is a just a simple, yet useful, reformulation of the definition of the birth-and-assassination process. Without further notice, we will assume that the birth-and-assassination process is the process $(Y_{t})_{t \geq 0}$ appearing in Theorem \ref {thm:coupl2}.  
In particular, as for the chase-escape process, by letting $t \rightarrow \infty$, we see that the total progeny of this process is related to the total progeny of a certain killed branching walk:

\begin {cor} \label {cor:BA}There exists a (random) time such that that no individuals remain in the birth-and-assassination process if and only if there is a finite number of individuals in  the branching random walk on $ \mathcal{U}$, produced with the point process $$  \delta_{ \{ \mathcal{K}_{ \emptyset}-\Exp_{1}( \lambda)\} }+    \delta_{ \{ \mathcal{K}_{\emptyset}-(\Exp_{1}( \lambda)+\Exp_{2}( \lambda))\} } +  \delta_{ \{ \mathcal{K} _{\emptyset}-(\Exp_{1}( \lambda)+\Exp_{2}( \lambda)+\Exp_{3}( \lambda))\}} +\cdots, $$
starting from $0$ and killed at $0$.
\end {cor}

Following Aldous \& Krebs, we say the process is \emph{stable}  if almost surely there exists a (random) time when no individuals remain in the system, and \emph{unstable}  otherwise. Let $ \phi(u)= \Es {e^{u K_{ \emptyset}}}$ be the moment generating function of $K_{ \emptyset}$ for $u \in \R$. Under the assumption that $ \phi$ is finite on a neighborhood of the origin, Aldous \& Krebs \cite {AK90} proved that if $ \min_{u>0} \lambda u^{-1} \phi(u)<1$ then the process is stable, and if $ \min_{u>0} \lambda u^{-1} \phi(u)>1$ then the process is unstable. Later, in the particular case where $ K_{ \emptyset}$ is an exponential random variable of parameter $1$, Bordenave \cite[Corollary 2]{Bor08} proved that the process is stable if  $\min_{u>0} \lambda u^{-1} \phi(u)=1$ (which corresponds to the case $ \la=1/4$). Our last contribution is to establish that this fact holds more generally.
 
\begin{thm}\label {thm:BA}Assume that $\phi$ is finite on a neighborhood of the origin and that $ \min_{u>0} \lambda u^{-1} \phi(u)=1$. Then the birth-and-assassination process is stable.\end{thm}

The proof is very similar to the one we give to Proposition \ref {prop:crit} and is also based on the almost sure convergence towards $0$ of the Biggins' martingale associated with our branching random walk.

\paragraph {Structure of the paper.} We prove our results concerning the chase-escape process in Section \ref {sec:CE}, and the results concerning the birth-and-assassination process in Section \ref {sec:BA}. Finally, technical uniform estimates on branching random walks are established in Section \ref {sec:appendix}.

\paragraph {Acknowledgments.} I am deeply indebted to Itai Benjamini for suggesting me to study this problem as well as to the Weizmann Institute of Science for hospitality, where this work begun. I would also like to thank Bastien Mallein and Elie Aïdékon for stimulating discussions. I am also grateful to Elie Aïdékon for the proof of Proposition \ref {prop:tails}.

\section{The chase-escape process}
\label {sec:CE}

In this section, we study the chase-escape process and prove in particular Proposition \ref {prop:crit} and Theorems \ref {thm:tail} and \ref {thm:tail2} by using branching random walks.

\subsection{Branching random walks}
 \label {sec:brw}
We start by introducing some relevant quantities of the branching random walk involved in the coupling with the chase-escape process.  For $x \geq 0$, let $ \P_{x}$ be a probability measure such that under $ \mathbb {P}_{x}$, $(V(u); u \in \mathcal{T})$  is the law of a branching random walk   produced with the point process $$ \mathcal{L}=    \sum_{i=1} ^ U \delta_{ \{  \mathcal{E}-\Exp_{i}( \lambda)\} }, $$
starting from $x$, where $U$ is a non-negative integer valued random variable distributed according to $ \nu$. We denote by  $ \mathbb {E}_{x}$ the corresponding expectation. Recall  that $d>1$ is the mean value of $ \nu$. A straightforward  computation yields the logarithmic generation function for the branching random walk:
  \begin{eqnarray*}
  \psi(t) &:=& \ln \mathbb {E}_{0} \left [ \sum_{|u|=1} e^{t V(u)} \right]=  \ln \left(  \sum_{k \geq 0}  \mathbb {E}_{0} \left [ \mathbbm {1}_{ \{ U=k\} } \sum_{i=1} ^{k} e^{t ( \mathcal{E}- \Exp_{i}( \lambda)) } \right] \right) \\
  &=&  \ln \left(  \sum_{k \geq 0}   \Pr {U=k} \mathbb {E}_{0} \left [ \sum_{i=1} ^{k} e^{t ( \mathcal{E}- \Exp_{i}( \lambda)) } \right] \right) =  \ln \left(  \sum_{k \geq 0}   \Pr {U=k}  k  \frac{ \lambda}{(1-t)( \lambda+t)}  \right) \\
  &=& \ln(d)+ \ln \left( \frac{1}{1-t} \cdot \frac{ \lambda}{ \lambda+t} \right), \qquad \qquad  t \in (- \lambda,1),
  \end{eqnarray*}
   In addition, there exists $ \rho_{ \star} \in (0,1)$ such that $ \psi( \rho_{ \star})= \rho_{ \star} \psi'( \rho_{ \star})$. Recalling that $\lambda_{c}= 2d-1- 2 \sqrt {d(d-1)}$, notice that
  $$ \psi'( \rho_{ \star}) \quad\begin {cases} \quad  <0 & \textrm { if } \lambda \in (0, \lambda_{c}) \\
  \quad =0 & \textrm { if }  \lambda= \lambda_{c}\\
  \quad >0 & \textrm { if } \lambda> \lambda_{c}. \end {cases}
  $$
 In the terminology of branching random walks, this means that the branching random walk $V$ is critical for $ \lambda= \lambda_{c}$, subcritical for $ \lambda \in (0, \lambda_{c})$ and supercritical for $ \lambda> \lambda_{c}$. For $ \lambda= \lambda_{c}$, note that $ \rho_{ \star}= 1+ \sqrt {d(d-1)}-d=(1- \lambda_{c})/2$.
 
 We are now ready to give an effortless proof of Proposition \ref {prop:crit}.
 
 \begin {proof}[Proof of Proposition \ref {prop:crit}.] Assume that $ \lambda= \lambda_{c}$, so that
 $ \psi'( \rho_{ \star})=0$. By  \cite [Lemma 5]{Bigg77} (see also \cite{Rus87} for a probabilistic proof), under $ \P_{0}$, the martingale $W_{n}:= \sum_{|u|=n}e^{  \rho_{ \star} V(u)}$ converges almost surely towards $0$ as $ n \rightarrow \infty$. This implies that, under $ \P_{0}$,
$$ \max_{|u|=n} V(u) \to - \infty \qquad  a.s.$$
(See \cite[Theorem 1.2]{HS09} for a more precise rate of convergence under additional assumptions on the moments of $ \nu$.)
This implies that the number of surviving individuals in the branching random walk $V$ killed at the barrier $0$  is almost surely finite. Proposition \ref {prop:crit}  then simply follows from Corollary \ref {cor:coupl1}.
 \end {proof}
 
 We finally introduce some notation in the subcritical case. When $ \lambda < \lambda_{c}$, we let $ \rho_{-}, \rho_{+}$ be such that $0 < \rho_{-}< \rho_{\star}< \rho_{+}<1$ and $ \psi( \rho_{-})= \psi( \rho_{+})=0$. Setting $ \Delta= \lambda^{2}-2 \lambda (2d-1)+1$, one checks that
  \begin{equation}
  \label{eq:rho} \rho_{-}= \frac{1}{2} \cdot (1- \lambda- \sqrt {\Delta}), \qquad \rho_{+}=\frac{1}{2} \cdot (1- \lambda + \sqrt {\Delta}).
  \end{equation}

\subsection{Total progeny of a killed branching random walk}

In this section, we assume that $ \lambda \leq \lambda_{c}$ and shall prove Theorem  \ref {thm:tail}. To this end, we begin by presenting results concerning the total progeny of branching random walks killed at the barrier $0$ starting from a fixed initial point. We first need to introduce an auxiliary random walk $(S_{n})_{n \geq 0}$ defined as follows. Set $ \rho= \rho_{ \star}$ if $ \lambda= \lambda_{c}$ (critical case) and $ \rho= \rho_{+}$ if $  \lambda \in (0, \lambda_{c})$ (subcritical case). Then, for every $x \geq 0$, let $(S_{n})_{n \geq 0}$ be a random walk such that, under the probability measure $ \mathbb {Q}_{x}$, we have $S_{0}=x$ and such that its step distribution is characterized by the fact that
$$ \EsQx {f(S_1-S_{0})}= \E_0 \left [ \sum_{|u|=1} f(V(u)) e^{ \rho V(u)} \right] \qquad  \textrm { for every measurable function } f: \R \rightarrow \R_{+}.$$
Next, let $R$ be the renewal function associated with $(S_{n})_{n \geq 0}$ defined by
$$R(x)= \mathbb{Q}_{0} \left[ \sum_{j=0}^{ \tau^{*}-1} \mathbbm {1}_{ \{ S_{j} \geq -x\} } \right], \qquad \textrm {where } \tau^{*}= \inf \{ j \geq 1; \ S_{j} \geq 0\},$$
and set $ \tau_{0}^{-}= \inf \{ k \geq 0 ; S_{k}<0\}$. 

We are now able to state a result obtained by Aïdékon, Hu \& Zindy \cite{AHZ13+}, which describes the asymptotic 
 behavior of the total progeny $Z$ of  the branching random walk killed at the barrier $0$ starting from a fixed initial point which is not necessarily the origin (this will be needed for the proof of Theorem \ref {thm:tail2}).
 
\begin{thm}[Theorem 1 in \cite{AHZ13+}] ~\label {thm:AHZ}
\begin{enumerate}
\item[(i)] Suppose that $ \lambda= \lambda_{c}$. Then for every fixed $ x \geq 0$ ,
$$ \P_{x} ( Z >n)  \quad\mathop{ \sim}_{n \rightarrow \infty} \quad  \frac{ \mathbb{Q}_{0}\left[e^{- \rho_{ \star} S_{ \tau_{0}^{-}}} \right] -1}{ d-1}  \cdot R(x)   e^{ \rho_{ \star}x} \cdot \frac{1}{n ( \ln(n))^{2}}.$$
\item[(ii)] Suppose that $ \lambda \in (0,\lambda_{c})$. There exists a constant $C_{1}>0$, depending only on $d$ and $ \lambda$, such that for every $x \geq 0$ and $n \geq 1$,
$$ \P_{x}(Z>n)  \quad\mathop{ \sim}_{n \rightarrow \infty} \quad C_{1} \cdot  R(x)   e^{ \rho_{+} x} \cdot n^{- \frac{ \rho_{+}}{ \rho_{-}}}.$$
\end{enumerate}
\end{thm}

\begin {proof}[Proof of Theorem \ref {thm:tail}]  ~
(i) Suppose that $ \lambda= \lambda_{c}$. We start by finding the step distribution of $(S_{n})_{n \geq 0}$ under $ \mathbb{Q}_{x}$.  If $f: \R \rightarrow \R_{+}$ is a measurable function, as for the calculation of the logarithmic generating function of $V$, write
\begin{eqnarray}
\mathbb{Q}_{0} \left[{f(S_1)} \right] &=&  \sum_{k \geq 0}   \Pr {U=k}   \sum_{i=1}^{k}\E_0 \left [ f( \mathcal{E} - \Exp_{i}( \lambda_{c})) e^{ \rho V( \mathcal{E} - \Exp_{i}( \lambda_{c}))} \right]  \notag\\
&=& \lambda_{c} d \int dx dy  \ \mathbbm {1}_{  \{x,y \geq 0\}} f(x-y) e^{ \rho_{ \star} (x-y)}e^{-x- \lambda_{c} y} =  \frac{ \lambda_{c} d}{ \lambda_{c}+1} \cdot  \int_{ \R} f(u)  e^{ - ( \la_{c}+1) |u|/2}du, \label{eq:jump}
\end{eqnarray}
where we have used the fact that $ \rho_{ \star}=(1- \lambda_{c})/2$ in the last equality. Hence the step distribution of $(S_{n})_{n \geq 1}$ is a symmetric two-sided exponential distribution. This implies that for $x \geq 0$, under $ \mathbb {Q}_{x}$, the random variable $-S_{ \tau_{0}^{-}}$ is distributed according to $ \Exp((1+ \lambda_{c})/2)$ (see e.g. Example (b) in \cite[Sec. VI.8]{Fel71}). In particular,
\begin{equation}
\label{eq:calc}\EsQ { e^{- \rho_{ \star} S_{ \tau_{0}^{-}}}}=  \frac{1+ \lambda_{c}}{2}  \cdot  \int_{0}^{ \infty} dx  \ e^{  \left( {1- \lambda_{c}} \right)  x /2 }e^{ -  \left( {1+ \lambda_{c}} \right) x/2 }= \frac{1+ \lambda_{c}}{2}  \cdot  \int_{0}^{ \infty} dx\ e^{- \lambda_{c} x} = \frac{1+ \lambda_{c}}{2 \lambda_{c}}.
\end{equation}
Hence, by Theorem \ref {thm:AHZ} (i),
\begin{equation}
\label{eq:equiv1} \P_{x} ( Z >n)  \quad\mathop{ \sim}_{n \rightarrow \infty} \quad  \frac{1}{ d-1}  \cdot   \left( \frac{1+ \lambda_c}{ 2 \lambda_c} -1 \right)     \cdot R(x)   e^{ \rho_{ \star}x}   \cdot \frac{1}{n ( \ln(n))^{2}}.\end{equation}
Assertion (i) then follows by taking $x=0$ in \eqref{eq:equiv1} and noting that $R(0)=1$. For the second assertion, it suffices to take $x=0$ in Theorem \ref {thm:AHZ} (ii).
\end {proof}

In the subcritical case $ \lambda \in (0, \lambda_{c})$, the expression of the constant $C_{1}$ appearing in Theorem \ref {thm:AHZ} (ii) is much more complicated than in the critical case: See in particular Lemma 1 in  \cite {AHZ13+} and the expression of the constant $c_{ \rho_{-}}$ in Eq.~(8.18) in  \cite {AHZ13+}, which arises in the asymptotic tail
behavior of the almost sure limit of the martingale $ \sum_{|u|=n} e^{  \rho_{-} V(u)}$ as $n \rightarrow \infty$. Unfortunately, we have not managed to find a simple expression for $C_{1}$ in our particular case.

\subsection{Total progeny of a killed branching random walk with a random starting point}

We keep here the notation introduced in the previous section. The proof of Theorem \ref {thm:tail2} requires the following uniform bounds on $ \P_{x} ( Z >n)$, which may be of independent interest.

\begin {prop} \label {prop:tails} There exists a constant $C_{2}( \lambda,d)>0$ such that the following assertions hold.
\begin{enumerate}
\item[(i)] Suppose that $ \lambda= \lambda_{c}$. For every $ x \geq 0$ and $n \geq 1$,
$\displaystyle \P_{x} ( Z >n) \leq C_{2}( \lambda_{c},d) \cdot (x+1)e^{ \rho_{ \star}x} \cdot  \frac{1}{n ( \ln(n))^{2}}$.
\item[(ii)] Suppose that  $ \lambda \in (0,\lambda_{c})$. For every $x \geq 0$ and $n \geq 1$,
$\displaystyle \P_{x}(Z>n) \leq C_{2}( \lambda,d) \cdot  e^{ \rho_{+} x} \cdot n^{- \frac{ \rho_{+}}{ \rho_{-}}}$.
\end{enumerate}
\end {prop}
We postpone the proof of Proposition \ref {prop:tails} to Section \ref {sec:appendix}.

\begin {proof}[Proof of Theorem \ref {thm:tail2}] Recall that  $ \mathcal{A} \subset \mathcal{T}$ is a finite connected subset of $ \mathcal{T}$ containing $ \emptyset$. To simplify notation,  denote by $Z( \mathcal{T}, \mathcal{A} ,x)$ the total number of infected individuals in the chase-escape process $ \textsf{C}( \mathcal{T}, \mathcal{A} ,x)$. For $ u \in  \partial \mathcal{A}$,  let $R_{u}$ be the first time when $ \overleftarrow {u}$ recovers in the chase-escape process $ \textsf{C}( \mathcal{T},  \mathcal{A} ,0)$. In particular, $R_{u}$ has the same distribution as the sum of $|u|$ independent exponential random variables of parameter $1$. Let also $(\mathcal{E}_{u}; u \in  \partial \mathcal{A})$ be a collection of  independent exponential random variables of parameter $ \lambda$, independent of $(R_{u}; u \in  \partial A)$. If $u \in \partial \mathcal{A}$, without the presence of recovered vertices, $  \overleftarrow {u}$ would infect $u$ after a time distributed as $ \mathcal{E}_{u}$. Hence, if $ \mathcal{T}_{u}$ denotes the tree of descendants of $u$ (including $u$), then  the number of individuals of $ \mathcal{T}_{u}$  that will be infected in the chase-escape process $ \textsf{C}( \mathcal{T},  \mathcal{A} ,0)$ has the same distribution as $Z( \mathcal{T}_{u} ,  \{  \emptyset\} , R_{u}-\mathcal{E}_{u}) \mathbbm {1}_{ \{ R_{u}> \mathcal{E}_{u} \} }$, where the random variables $(Z( \mathcal{T}_{u} ,  \{  \emptyset\} , R_{u}-\mathcal{E}_{u}) \mathbbm {1}_{ \{ R_{u}> \mathcal{E}_{u} \} }; u \in \partial \mathcal{A})$ are independent conditionnally on $(R_{u}; u \in \partial \mathcal{A})$, and in addition this equality holds jointly in distribution for all $u \in \partial \mathcal{A}$. Now let $u_{1}, \ldots, u_{K}$ be an enumeration of the vertices of $ \partial \mathcal{A}$.  By the previous discussion, we have
\begin{eqnarray}
\Pr {Z_{ \mathcal{A} }>n \ \big| \ R_{u_{1}}, \ldots ,R_{u_{K}}} &=& \Pr { |\mathcal{A}| + \sum_{i=1}^{K} Z( \mathcal{T}_{u_{i}} ,  \{  \emptyset\} , R_{u_{i}}-\mathcal{E}_{u_{i}}) \mathbbm {1}_{ \{ R_{u_{i}}> \mathcal{E}_{u_{i}} \} } >n \  \big | \ R_{u_{1}}, \ldots ,R_{u_{K}}} \label {eq:affr}
\end{eqnarray}
where by denote by $| \mathcal{A}|$ the cardinal of the set $ \mathcal{A}$. Moreover, by Corollary \ref {cor:coupl1},
$$ \Pr {Z( \mathcal{T}_{u_{i}} ,  \{  \emptyset\} , R_{u_{i}}-\mathcal{E}_{u_{i}}) \mathbbm {1}_{ \{ R_{u_{i}}> \mathcal{E}_{u_{i}} \} }>n  \ \big| \ R_{u_{1}}, \ldots ,R_{u_{K}} } = \Es { \P_{R_{u_{i}} -\mathcal{E}_{u_{i}}}(Z>n)   \mathbbm {1}_{ \{ R_{u_{i}}> \mathcal{E}_{u_{i}} \} }  \ \big| \ R_{u_{i}} }$$
for every $1 \leq i \leq K$.

We start by proving (i), where $ \lambda= \lambda_{c}$. Recall from \eqref{eq:jump} that under $ \Q_{x}$, $(S_{n})_{n \geq 0}$ is a random walk with step distribution a symmetric two-sided exponential distribution.  The renewal function $R$ is thus explicit: $R(x)=1+ (1+ \la_{c}) x/2$ for $x \geq 0$ (see e.g. Eq.~(4.3) of Chapter XII.4 in \cite{Fel71}).  Hence by Theorem \ref {thm:AHZ} (i) and \eqref{eq:calc}, 
\begin{eqnarray*}
&& n ( \ln(n))^{2} \cdot \P_{R_{u_{i}} -\mathcal{E}_{u_{i}}}(Z>n)   \mathbbm {1}_{ \{ R_{u_{i}}> \mathcal{E}_{u_{i}} \} } \\
&& \qquad \qquad  \quad  \mathop { \longrightarrow}_{n \rightarrow \infty} \quad \frac{1- \lambda_c}{ 2 \lambda_c(d-1)}     \cdot  \left( 1+  \frac{(1+ \la_{c}) (R_{u_{i}}-\mathcal{E}_{u_{i}})}{2} \right)    e^{ \rho_{ \star}(R_{u_{i}}-\mathcal{E}_{u_{i}})}    \cdot  \mathbbm {1}_{ \{ R_{u_{i}}> \mathcal{E}_{u_{i}} \} }.
\end{eqnarray*}
Proposition \ref {prop:tails} (i) allows us to use the conditioned dominated convergence theorem to deduce that
 \begin{eqnarray*}
  \Es { \P_{R_{u_{i}} -\mathcal{E}_{u_{i}}}(Z>n)   \mathbbm {1}_{ \{ R_{u_{i}}> \mathcal{E}_{u_{i}} \} }
 \ \big| \ R_{u_{i}} }  & \displaystyle   \mathop {\sim}_{n \rightarrow \infty} & G(R_{u_{i}}) \cdot \frac{1}{n ( \ln(n))^{2}},
 \end{eqnarray*}
 where $G(x)$ is defined by
 $$G(x)= \frac{1- \lambda_c}{ 2 \lambda_c(d-1)}          \cdot  \Es{ \left( 1+  \frac{(1+ \la_{c}) (x-\mathcal{E}_{u_{i}})}{2} \right)    e^{ \rho_{ \star}(x-\mathcal{E}_{u_{i}})}   \mathbbm {1}_{ \{ x> \mathcal{E}_{u_{i}} \} }}, \qquad x \geq 0.$$
 But if $X_{1}, \ldots, X_{K}$ are independent random variables such that $ \Pr {X_{i}>n} \sim  \kappa_{i} /(n \ln(n)^{2})$ as $n \rightarrow \infty$, since $1/ \ln(n)^{2}$ is regularly varying, then $ \Pr {X_{1}+ \cdots + X_{K}>n} \sim (\kappa_{1}+ \kappa_{2} + \cdots  \kappa_{K})  /(n \ln(n)^{2})$  (see e.g. \cite[Proposition in Sec. VIII.8]{Fel71}). Thus
$$\Pr { |\mathcal{A}| + \sum_{i=1}^{K} Z( \mathcal{T}_{u_{i}} ,  \{  \emptyset\} , R_{u_{i}}-\mathcal{E}_{u_{i}}) \mathbbm {1}_{ \{ R_{u_{i}}> \mathcal{E}_{u_{i}} \} } >n \ \big| \ R_{u_{1}}, \ldots ,R_{u_{K}}} \quad  \mathop {\sim}_{n \rightarrow \infty} \quad   \left( \sum_{i=1}^{K} G(R_{u_{i}})  \right)  \cdot \frac{1}{ n \ln(n)^{2}}.$$
 Observing that
 \begin{eqnarray*}
 && \Pr { |\mathcal{A}| + \sum_{i=1}^{K} Z( \mathcal{T}_{u_{i}} ,  \{  \emptyset\} , R_{u_{i}}-\mathcal{E}_{u_{i}}) \mathbbm {1}_{ \{ R_{u_{i}}> \mathcal{E}_{u_{i}} \} } >n \ \big| \ R_{u_{1}}, \ldots ,R_{u_{K}}}  \\
 && \qquad \qquad \leq    \sum_{i=1}^{K} \Pr { Z( \mathcal{T}_{u_{i}} ,  \{  \emptyset\} , R_{u_{i}}-\mathcal{E}_{u_{i}}) \mathbbm {1}_{ \{ R_{u_{i}}> \mathcal{E}_{u_{i}} \} } > (n-|\mathcal{A}|)/K  \ \big| \ R_{u_{1}}, \ldots ,R_{u_{K}} }  \\
 && \qquad \qquad =   \sum_{i=1}^{K} \Es { \P_{R_{u_{i}} -\mathcal{E}_{u_{i}}}(Z>(n-|\mathcal{A}|)/K)   \mathbbm {1}_{ \{ R_{u_{i}}> \mathcal{E}_{u_{i}} \} }
 \ \big| \ R_{u_{i}} },
 \end{eqnarray*}
we combine once again Proposition \ref {prop:tails} (i) with the dominated convergence theorem, to get, using \eqref{eq:affr}, that 
$$ \Pr {Z_{ \mathcal{A} }>n} =  \Es {\Pr {Z_{ \mathcal{A} }>n \ \big| \ R_{u_{1}}, \ldots ,R_{u_{K}}}}  \quad  \mathop {\sim}_{n \rightarrow \infty} \quad    \left( \sum_{i=1}^{K}  \Es{G(R_{u_{i}})}  \right)  \cdot \frac{1}{ n \ln(n)^{2}}.
 $$
  It hence remains to compute  $\Es{G(R_{u_{i}})}$. Recall that $R_{u_{i}}$ has the same distribution as the sum of $|u_{i}|$ independent exponential random variables of parameter $1$, and to simplify notation set $k=|u_{i}|$. Then using a change of variables and Fubini's theorem, write for every measurable function  $F : \R \rightarrow \R_{+}$:
  \begin{eqnarray}
  \Es {F( R_{u_{i}}- \mathcal{E}_{u_{i}}) \mathbbm {1}_{ \{ R_{u_{i}}> \mathcal{E}_{u_{i}} \} }}&=& \int  dx dy \ F(x-y) \frac{x^{k-1} e^{-x}}{ \Gamma(k)} \lambda e^{- \lambda y} \mathbbm {1}_{ \{ x>y \geq 0\} } \notag\\
  &=&   \int_{0}^{ \infty} du \ \frac{ \lambda u^{k-1} e^{-u(1+ \lambda)}}{ \Gamma(k)} \int_{0}^{u} F(v) e^{ \lambda v} dv. \label {eq:fub}
  \end{eqnarray}
  Recall that $\rho_{ \star}=(1- \lambda_{c})/2$. Since $ \int_{0}^{u} (1+ (1+ \lambda_{c})v/2) e^{(1- \lambda_{c}) v/2} e^{ \lambda_{c} v} dv=u \cdot e^{(1+ \lambda_{c})u/2}$, a straightforward calculation yields
  $$\Es{G(R_{u_{i}})}=\Es {F( R_{u_{i}}- \mathcal{E}_{u_{i}}) \mathbbm {1}_{ \{ R_{u_{i}}> \mathcal{E}_{u_{i}} \} }}= \frac{1- \lambda_c}{ 2 \lambda_c(d-1)}  \cdot \int_{0}^{ \infty} du \ \frac{ \lambda_c u^{k-1} e^{-u(1+ \lambda_c)}}{ \Gamma(k)} u \cdot e^{(1+ \lambda_{c})u/2}.$$
  Hence
  $$\Es{G(R_{u_{i}})}=
  \left( 1+ \sqrt{\frac{ d}{d-1}} \right)  \cdot  \frac{ \lambda_{c} |u_{i}|   }{((1+ \lambda_{c})/2)^{1+|u_{i}|}}=   \left( 1+ \sqrt{\frac{ d}{d-1}} \right)  \cdot\left(   \frac{|u_{i}|}{d \cdot (d- \sqrt {d(d-1)})^{|u_{i}|-1}} \right).$$
  This completes the proof of (i).
  
  The proof of the second assertion, where $ \lambda \in (0,  \lambda_{c})$ is very similar, so we only sketch the main calculations. Under $ \Q_{x}$, $(S_{n})_{n \geq 0}$ is still a random walk with step distribution a (non-symmetric) two-sided exponential distribution, with density $ \frac{ \lambda d}{d+1} e^{-(1- \rho_{+}) u}$ on $ \R_{+}$ and density $\frac{ \lambda d}{d+1} e^{( \lambda+ \rho_{+}) u}$ on $ \R_{-}$. By 
  \cite[Example (a) (ii) in XII.4]{Fel71}, the renewal function $R$ is equal to 
  $$R(x)= \frac{ \rho_{+}+ \lambda}{2 \rho_{+}+ \lambda-1}+ \frac{ \rho_{+}-1}{2 \rho_{+}+ \lambda-1} e^{- (2 \rho_{+}+ \lambda-1) x}.$$
 (In our case, the value of $ \kappa$ defined in \cite[Example (a) (ii) in XII.4]{Fel71} is $ \kappa= 2 \rho_{+}+ \lambda-1$.) Hence
 $$ \Pr {Z_{ \mathcal{A} }>n}  \quad  \mathop {\sim}_{n \rightarrow \infty} \quad    \left( \sum_{i=1}^{K}  \Es{H(R_{u_{i}})}  \right)  \cdot n^{- \frac{ \rho_{+}}{ \rho_{-}}}
 $$
 where $H(x)$ is defined by
 $$H(x)= C_{1}   \cdot  \Es{ \left(  \frac{ \rho_{+}+ \lambda}{2 \rho_{+}+ \lambda-1}+ \frac{ \rho_{+}-1}{2 \rho_{+}+ \lambda-1} e^{- (2 \rho_{+}+ \lambda-1) (x-\mathcal{E}_{u_{i}} )}  \right) e^{ \rho_{ +}(x-\mathcal{E}_{u_{i}})} \mathbbm {1}_{ \{ x> \mathcal{E}_{u_{i}} \} }}, \qquad x \geq 0.$$
We compute $\Es{H(R_{u_{i}})}$  by using \eqref{eq:fub}: Since  $ \int_{0}^{u} R(v) e^{ \rho_{+} v}  \cdot e^{ \lambda v} dv= (e^{ ( \lambda+ \rho_{+})u} - e^{u(1- \rho_{+})})/( 2 \rho_{+}+ \lambda-1)$, we get 
$$\Es{H(R_{u_{i}})}= C_{1} \frac{ \lambda \left(  (1- \rho_{+}) ^{-k}-( \lambda+ \rho_{+})^{-k} \right) }{2 \rho_{+}+ \lambda-1}=C_{1} \cdot     \frac{ \lambda}{ \sqrt {\lambda^{2}-2 \lambda (2d-1)+1}}  ( (\rho_{-}+ \lambda)^{-|u|}- (\rho_{+}+ \lambda)^{-|u|}).$$
This completes the proof.
\end {proof}

Observe that this proof shows that Theorem \ref {thm:tail2} holds when $ \mathcal{T}$ is a $ \nu$-Galton--Watson tree (under the same integrability assumption as in Theorem \ref {thm:tail}) and when $ \mathcal{A}$ is chosen in such a way that the trees $( \mathcal{T}_{u}; u \in \partial \mathcal{A})$ are independent $ \nu$- Galton--Watson trees.

\subsection{Reaching high generations}

In this section, we state a conjecture concerning the asymptotic behavior of the probability that the infection reaches high levels of the tree. Assume that $ \nu(d)=1$ where $d \geq 1$ is an integer, so that  $ \mathcal{T}$ is the infinite $d$-ary tree. Recall that $Z_{n}$ denotes the number of individuals of the $n$-th level of the tree $ \mathcal{T}$  that have been infected in the chase-escape process $ \textsf{C}( \mathcal{T},  \{ \emptyset\},0)$. Finally set $ \gamma= 4 \lambda/(1+ \lambda)^{2}$. 

We believe that the following results hold:
\begin{enumerate}
\item[(i)] If $ \lambda \in (0, \la_{c})$, there exists a constant $C_{3}>0$, depending only on $ \lambda$ and $d$, such that 
 $ \Pr{Z_{n}>0} \sim C_{3} \cdot (\gamma d)^n  n^{-3/2}$ as $n \rightarrow \infty$.
\item[(ii)]  If $ \lambda=\lambda_{c}$, we have  $  \ln \Pr {Z_{n}>0} \sim  -  \left( 3 \left( 1- {1}/{d} \right) \pi^{2} \right)^{1/3} \cdot  n^{1/3}$  as $n \rightarrow \infty$.\end{enumerate}

Aïdékon \& Jaffuel \cite {AJ11} proved this result for branching random walks on $ \mathcal{T}$ with i.i.d.~displacements. However, we believe that even if in our case the displacements are not independent, an analog result should hold, and that their proof could be adapted (see in particular \cite{Jaf12}, where the displacements are not supposed to be independent). We also believe that this should hold more generally for  Galton--Watson trees under adequate integrability conditions.

\section{Birth-and-assassination process}
 \label {sec:BA}
 
 We are now interested in the asymptotic behavior of the birth-and-assassination process at criticality.
Denote by $ \widetilde { \P}$ the law of the branching random walk $(\widetilde{V}(u); u \in \mathcal{U})$ on $ \mathcal{U} = \cup_{n \geq 0} \N^{n}$ produced with the point process $$ \widetilde{\mathcal{L}} =   \delta_{ \{ \mathcal{K}_{ \emptyset}-\Exp_{1}( \lambda)\} }+    \delta_{ \{ \mathcal{K}_{\emptyset}-(\Exp_{1}( \lambda)+\Exp_{2}( \lambda))\} } +  \delta_{ \{ \mathcal{K} _{\emptyset}-(\Exp_{1}( \lambda)+\Exp_{2}( \lambda)+\Exp_{3}( \lambda))\}} +\cdots, $$
where $(\Exp_{i}( \lambda))_{i \geq 1}$ is an i.i.d.~collection of exponential random variables of parameter $ \lambda$, with starting point $0$ and independent of $ K_{ \emptyset}$.

Introduce the logarithmic generation function $ \widetilde { \psi}(t)=\ln \E_{0}  \left[ \sum_{|u|=1} e^{t \widetilde{V}(u)} \right] $  of this branching random walk. Recalling that $ \phi(t)= \Es { e^{t\mathcal{K}_{ \emptyset} }}$ is the moment-generating function of $ \mathcal{K}_{ \emptyset}$, a straightforward calculation yields
$$ \widetilde {\psi}(t)=  \ln \left(  \sum_{i=1}^{ \infty} \Es {e^{t \mathcal{K}_{\emptyset} }}  \Es {e^{-t \Exp_{1}( \lambda)}}^{i} \right)= \ln \left(   \phi(t) \sum_{i=1}^{ \infty} \left( \frac{ \la}{ \la+t} \right)^{i} \right)= \ln \left(  \frac{ \la \phi(t)}{t} \right), \qquad t \geq 0.$$

\begin {proof}[Proof of Theorem \ref {thm:BA}] 
The proof is very similar to that one we gave for Proposition \ref {prop:crit}: assuming that $\min_{u>0} \lambda u^{-1} \phi(u)=1$ and that this minimum is attained at $u_{ \star}>0$, we have $\widetilde { \psi}'(u_{ \star})=0$, so that  by  \cite [Lemma 5]{Bigg77} (see also \cite{Rus87}), the martingale $W_{n}:= \sum_{|u|=n}e^{ u_{ \star}\widetilde{V}(u)}$ converges almost surely towards $0$ as $ n \rightarrow \infty$. This implies that under $\widetilde { \P}$,
$$ \max_{|u|=n} \widetilde{V}(u) \to - \infty \qquad  a.s.$$
Hence the number of surviving individuals in the branching random walk $\widetilde{V}$ killed at the barrier $0$. Theorem \ref {thm:BA} then simply follows from Corollary \ref {cor:BA}.
\end {proof}

Now assume that $K_{ \emptyset}$ is an exponential random variable of parameter $1$, which implies that $ \widetilde {\psi}(t)= \ln \left( { \la} /(t(1-t)) \right)$, so that the critical parameter is $ \lambda=1/4$. In this case, if $N$ denotes the total number of born individuals in the birth-and-assassination process, using analytical methods Bordenave \cite[Theorem 2]{Bor08} showed that for every $ \lambda \in (0,1/4)$,
\begin{equation}
\label{eq:B}\sup \{ u \geq 0;  \Es { N^{u}}< \infty\}= \frac{1+ \sqrt {1- 4 \lambda}}{1- \sqrt {1-4 \lambda}}
\end{equation}and that $ \Es {N}=2$ for $ \lambda=1/4$. In this case, and also when $K_{ \emptyset}$ is any positive random variable, it would be interesting to find whether similar phenomena as in the chase-escape process occur for the tail of the total progeny and the probability of having individuals born at large generations occur in the birth-and-assassination process. Unfortunately, we may not apply known results on killed branching random walks since here on the one hand each individual has an infinite number of offspring (the branching is infinite) and on the other hand the logarithmic generation function $ \widetilde { \psi}$ is not finite on a neighborhood of the origin. However, it seems likely that this is not an obstacle to have an analog of Theorem \ref{thm:AHZ} to hold in this case as well. It is plausible that, in general, if $\min_{u>0} \lambda u^{-1} \phi(u)=1$, then $ \Pr{N>n} \sim C/ ( n \ln(n)^2)$ as $n \rightarrow \infty$ for a certain $C>0$, and  if $\min_{u>0} \lambda u^{-1} \phi(u)<1$, then $ \Pr{N>n}  \sim  C  \cdot n^{- \widetilde{ \rho}_{+}/ \widetilde{ \rho}_{-} } $ for a certain $C>0$, where $ 0< \widetilde{ \rho}_{-} < \widetilde{ \rho}_{+}$ are such that 
$$\lambda \frac{\phi(\widetilde{ \rho}_{-})}{\widetilde{ \rho}_{-}} = \lambda \frac{\phi(\widetilde{ \rho}_{+})}{ \widetilde{ \rho}_{+}} =1.$$
In this case, notice that by convexity of $t \mapsto \Es{ e^{t \mathcal{K}_{ \emptyset} }}/t$, there exist exactly two values of $t>0$ such that $\lambda t^{-1} \phi(t)=1$. In addition, when $K_{ \emptyset}$ is an exponential random variable of parameter $1$, we have $ \widetilde{ \rho}_{ \pm}=(1\pm \sqrt{1-4 \lambda})/2$, so that this prediction is consistent with \eqref{eq:B}.

\section {Proof of the technical estimate}
\label {sec:appendix}

We first introduce some notation which appears in \cite{AHZ13+}. We shall refer the reader very often to \cite {AHZ13+} to stay as concise as possible.

Let $ \mathcal{L}[0]$ be the set of all the individuals of the (non-killed) branching random walk which lie below $0$ for its first time, see  \cite[Eq.~(1.8) and Fig.~1]{AHZ13+}. For $L>0$, let $H(L)$ be the number of  individuals of the branching random walk on $[0,L]$ with two killing barrier which were absorbed at level $L$ (see \cite[Eq.~(1.10) and Fig.~2]{AHZ13+}). Then let $Z[0,L]$ be the number of individuals of  $\mathcal{L}[0]$ which have not crossed level $L$ (see \cite[Eq.~(1.12)]{AHZ13+}). These individuals are partitioned into  \emph{good} and \emph{bad} individuals, whose number is denoted respectively by $Z_{g}[0,L]$ and $Z_{b}[0,L]$ (and taking the parameter $ \lambda=1$ appearing in their definition, see \cite[Eq.~(7.2)]{AHZ13+} for the critical case and \cite[Eq.~(8.5)]{AHZ13+} for the subcritical case). Finally, for a random walk $(S_{n})_{n \geq 0}$ set $ \tau_{a}^{+}= \inf \{ k \geq 0; S_{k}>a\}$ and $ \tau_{a}^{-}= \inf \{ k \geq 0; S_{k}<a\}$ for $a \in \R$.  

\begin {proof}[Proof of Proposition \ref {prop:tails}] The following proof is due to Elie Aïdékon. A close inspection of the proof of Lemma 2 of \cite {AHZ13+} shows that it is sufficient to establish Proposition \ref {prop:tails} when $ \Prx {Z>n}$ is replaced by $ \Prx { \#\mathcal{L}[0]>n}$. In the sequel $C$ denotes a positive constant which may change from line to line.
 
 We first concentrate on the critical case $ \la= \la_{c}$. By a linear transformation on $V$, we may assume that $ \rho_{ \star}=1$. We first claim that that there exists a constant $C>0$ such that for every $L > 0$ and $x \geq 0$, $\Esx{H(L)} \leq C {e^{x}(1+x) e^{-L}}/{L}$. It is enough to check this inequality for $0<x<L$. In this case, using Proposition 3 in \cite {AHZ13+} (applied to $ \mathcal{C}_{L}$, $ \lambda=1$ and $h(u)=e^{V(u)}$ with $ \mathbb {Q}_{x}$ defined by \cite[Eq.~(5.16)]{AHZ13+}), we have
 $$\Esx{H(L)} =e^{  x} \mathbb {Q}_{x} \left[ e^{ -\tau_{L}^{+}} \mathbbm {1}_{  \{ \tau_{L}^{+} < \tau_{0}^{-}\} }\right] \leq e^{x-L}\mathbb {Q}_{x}( \tau_{L}^{+}< \tau_{0}^{-}) \leq   e^{x-L} \frac{x+C}{L},$$
 where we have used \cite[Eq.~(4.12)]{AHZ13+} for the last inequality. This yields our claim. In particular, taking $L=L_{n}= \ln n+ \ln \ln n$,  we get
 \begin{equation}
 \label{eq:t2} \Prx {H(L_{n})) \geq 1} \leq C  (1+x)e^{x} \frac{1}{n (\ln(n))^{2}}.
 \end{equation}
 Next, by Lemma 13 in \cite{AHZ13+} we have $ \Esx {Z_{b}[0,L_n]} \leq C (1+x) e^{x} \cdot   (\ln(n))^{-2}$ and  by Lemma 14 in \cite{AHZ13+} we have $ \Esx {Z_{g}[0,L_n]^{2}} \leq C (1+x) e^{x}  \cdot  n(\ln(n))^{-2}$. Hence using Markov's inequality we get
 \begin{equation}
 \label{eq:t3} \Prx {Z_{b}[0,L_n] \geq n/2}  \leq C  (1+x)e^{x}  \cdot \frac{1}{n (\ln(n))^{2}} ,\qquad \Prx {Z_{g}[0,L_n] \geq n/2}  \leq C  (1+x)e^{x}  \cdot \frac{1}{n (\ln(n))^{2}}.
 \end{equation}
  The conclusion then follows  by observing that $ \Prx { \#\mathcal{L}[0]>n} \leq   \Prx {H(L_{n}) \geq 1}+  \Prx {Z[0,L_n]>n}$ and using \eqref {eq:t2} and \eqref {eq:t3}.
  
  We now turn to the subcritical case. Using  Proposition 3 in \cite {AHZ13+} (applied to $ \mathcal{C}_{L}$, $ \lambda=1$ and $h(u)=e^{ \rho_{-}  V(u)}$ with $ \mathbb {Q}^{( \rho_{-})}_{x}$ defined in the beginning of Sec.~8 in \cite{AHZ13+}) and Eq.~(8.1) in \cite{AHZ13+}, one similarly shows that
  $$ \Esx {H(L)}= e^{ \rho_{-}x}  \mathbb {Q}^{( \rho_{-})}_{x} \left[ e^{ - \rho_{-}\tau_{L}^{+}} \mathbbm {1}_{  \{ \tau_{L}^{+} < \tau_{0}^{-}\} }\right] \leq e^{  \rho_{-}x-\rho_{-}L}\mathbb {Q} ^{(\rho_{-})}_{x}( \tau_{L}^{+}< \tau_{0}^{-}) \leq  e^{ \rho_{-} x- \rho_{+} L}.$$
  Hence, taking $L=L_{n}= \ln(n)/ \rho_{-}$ we get $ \Prx {H(L_{n}) \geq 1} \leq e^{ \rho_{-} x} n^{- \rho_{+}/ \rho_{-}}$. By Lemma 19 in \cite{AHZ13+},  $ \Esx {Z_{g}[0,L_{n}]^{k_{ \star}}} \leq C e^{ \rho_{+}x}  \cdot n^{k_{ \star}- \rho_{+}/ \rho_{-}}$ with $k_{ \star}= \lfloor \rho_{+}/ \rho_{-} \rfloor+1$, and by Lemma 20 (i) in \cite{AHZ13+}, $ \Esx {Z_{b}[0,L_{n}]} \leq C e^{ \rho_{+} x} \cdot n^{1- \rho_{+}/ \rho_{-}}$. Markov's inequality then entails
  $$\Prx {Z_{b}[0,L_n] \geq n/2}  \leq C e^{ \rho_{+} x} \cdot n^{- \rho_{+}/ \rho_{-}},\qquad \Prx {Z_{g}[0,L_n] \geq n/2}  \leq C e^{ \rho_{+} x} \cdot n^{- \rho_{+}/ \rho_{-}}.$$
  We conclude as in the critical case.
\end {proof}

\bibliographystyle{siam}

\end{document}